\documentclass{amsart}

\usepackage{amssymb, amsmath,mathtools,mathabx}
\usepackage{mathrsfs}
\usepackage{amscd}
\usepackage{verbatim}
\usepackage{stmaryrd}
\usepackage[dvipsnames]{xcolor}
\usepackage{a4wide}
\usepackage{tikz}
\usetikzlibrary{shapes.geometric, arrows}
\usepackage{stmaryrd}

\usepackage{enumerate}

\mathtoolsset{showonlyrefs}
\numberwithin{equation}{section}

\usepackage[colorlinks,linkcolor={blue},citecolor={blue},urlcolor={purple},]{hyperref}

\usepackage[colorinlistoftodos,prependcaption,textsize=tiny]{todonotes}

\usepackage{forest}
\usepackage{tikz}
\tikzset{>=latex}

\usepackage{tabularx}
\newcolumntype{L}{>{\arraybackslash}X}
\usepackage{multirow}

\theoremstyle{plain}
\newtheorem{theorem}{Theorem}[section]
\theoremstyle{remark}
\newtheorem{remark}[theorem]{Remark}
\newtheorem{example}[theorem]{Example}
\theoremstyle{plain}

\numberwithin{equation}{section}

\def\N{{\mathbb N}}

\def\R{{\mathbb R}}

\newcommand{\E}{{\mathbf E}}

\newcommand{\F}{{\mathscr F}}
\newcommand{\Do}{\mathsf{D}}

\newcommand{\dd}{\mathrm{d}}
\newcommand{\Dom}{\mathcal{O}}
\newcommand{\embed}{\hookrightarrow}
\newcommand{\A}{\mathcal{A}}
\newcommand{\B}{\mathcal{B}}
\newcommand{\W}{\beta}

\allowdisplaybreaks

\begin{document}

\author{Antonio Agresti}
\address{Department of Mathematics Guido Castelnuovo, Sapienza University of Rome,
P.le Aldo Moro 5, 00185 Rome, Italy}
\email{antonio.agresti@uniroma1.it}

\author{Mark Veraar}
\address{Delft Institute of Applied Mathematics\\
Delft University of Technology \\ P.O. Box 5031\\ 2600 GA Delft\\The
Netherlands} \email{M.C.Veraar@tudelft.nl}

\thanks{The first author is a member of GNAMPA (IN$\mathrm{d}$AM). The second author has received funding from the VICI subsidy VI.C.212.027 of the Netherlands Organisation for Scientific Research (NWO)}

\date\today

\title[A threshold for temporal regularity of stochastic PDE\lowercase{s}]{A Note on a threshold for temporal regularity\\ of stochastic PDE\lowercase{s}}



\begin{abstract}
We consider solutions to linear parabolic SPDEs of the form
$$
\dd u(t) + A u(t)\,\dd t = g(t)\, \dd \W, \qquad u(0)=0,
$$
where $A$ is a positive, invertible, and self-adjoint operator on a Hilbert space $X$, $\W$ is a one-dimensional Brownian motion, and $g(t)\equiv x\in X$. We show that, for all $\alpha\in [0,\frac{1}{2})$, 
$$
u\in L^2(\Omega;W^{\alpha,2}(0,T;\Do(A^{1/2}))) \quad \text{ if and only if }\quad x\in \Do(A^{\alpha}).
$$ 
In particular, there is a lack of persistence of temporal regularity from the diffusion coefficient $g$ to the solution, and additional spatial regularity is required to improve time regularity. In particular, this provides a counterexample to a conjectured time-regularity property for monotone stochastic evolution equations posed by D.\ Breit and M.\ Hofmanov\'a in [C.\ R.\ Math.\ Acad.\ Sci.\ Paris 354 (2016), 33–37].
\end{abstract}

\maketitle

\section{Introduction}
In this paper, we consider parabolic stochastic evolution equations of the form
\begin{equation}\label{eq:SEE}
\left\{
\begin{aligned}
&\dd u(t)   +A u(t)\, \dd t = g(t)\, \dd \W(t) \quad \text{ on }[0,T],
\\ & u(0) = 0,
\end{aligned}
\right.
\end{equation}
where $A$ is a positive, invertible, and self-adjoint operator on a Hilbert space $X$, $\W=(\W(t))_{t\geq 0}$ is a one-dimensional Brownian motion, $g$ is a progressively measurable process with values in $X$ (or in a subspace thereof), and $T\in (0,\infty)$ is the time-horizon. 

It is a well-known fact in the theory of stochastic evolution equations 
\begin{equation}
\label{eq:zero_regularity_statement}
\E\|u\|_{L^2(0,T;\Do(A^{1/2}))}^2 \lesssim_A \E\|g\|_{L^2(0,T;X)}^2,
\end{equation}
where $\E$ denotes the expected value, and the domain of the square root $A^{1/2}$ can be defined via the spectral theorem for self-adjoint operators.
The estimate \eqref{eq:zero_regularity_statement} was first proven by Da Prato in \cite{DPZ82}, see also \cite[Theorem 6.12(2)]{DPZ} or \cite[Theorem 3.13]{AV25_survey}. 
It is natural to ask whether the regularity of the solution $u$ can be improved by assuming additional time regularity of $g$. Namely, if for some $\alpha\in [0,\frac{1}{2})$, the following variant of \eqref{eq:zero_regularity_statement} holds:
\begin{equation}
\tag{P}
\label{eq:alpha_regularity_statement}
\E\|u\|_{W^{\alpha,2}(0,T;\Do(A^{1/2}))}^2 \lesssim_{\alpha,A,T} \E\|g\|_{W^{\alpha,2}(0,T;X)}^2 ,
\end{equation}
where $W^{\alpha,2}(0,T;X)$ is the (vector-valued) Sobolev-Slobodeckij space defined as the set of all functions $v\in L^2(0,T;X)$ such that 
\begin{equation}
\label{eq:seminorm_Walpha}
[v]_{W^{\alpha,2}(0,T;X)}:=
\Big(\int_0^T\int_0^{T} \frac{\|v(t)-v(s)\|_X^2}{|t-s|^{1+2\alpha}}\,\dd t \,\dd s\Big)^{1/2}<\infty,
\end{equation}
endowed with the natural norm, see e.g., \cite[Section 2.5.d]{Analysis1}.
Of course, the threshold $\alpha<\frac{1}{2}$ is due to the pathwise regularity of Brownian motion. 

Establishing the validity of \eqref{eq:alpha_regularity_statement} is closely linked to the conjectured time-regularity for stochastic evolution equations in the monotone setting, as proposed by Breit and Hofmanov\'a \cite{BH16} (see also 
\cite[Subsection 1.4]{WichThesis}), motivated by numerical applications \cite{BHL21_numerics}. 
By using different techniques and additional assumptions, the time regularity result used in the latter work was proven in \cite{W23_siam}. In the current notes, we prove that such additional assumptions are also necessary; see the comments below Theorem \ref{thm:main2}.
The relationship between our framework and that considered in the aforementioned work is discussed in Subsection \ref{ss:variational_comments} below.

\smallskip

The main result of this manuscript shows that \eqref{eq:alpha_regularity_statement} does not hold for all $\alpha>0$. 

\begin{theorem}[Absence of temporal regularity persistency in parabolic SPDEs]
\label{thm:main}
Let $A$ be a positive, invertible, and (possibly unbounded) self-adjoint operator on a Hilbert space $X$. Assume that $g(t)\equiv x$ for some $x\in X$. Then, for all $\alpha\in [0,\frac{1}{2})$, 
$$
u\in L^2(\Omega;W^{\alpha,2}(0,T;\Do(A^{1/2})))\ \ \ \text{ if and only if }\ \ \  x\in \Do(A^{\alpha}).
$$
Moreover, the following norm equivalence holds
$$
\big(\E\|u\|_{W^{\alpha,2}(0,T;\Do(A^{1/2}))}^2\big)^{1/2}\eqsim_{\alpha,A,T} \|x\|_{\Do(A^{\alpha})}.
$$ 
In particular, \eqref{eq:alpha_regularity_statement} does not hold for any $\alpha>0$. 
\end{theorem}

The above result, in particular, shows that the conjectured result in \cite[Theorem 1.1]{BH16} cannot be valid under the assumptions stated therein; see Theorem \ref{thm:main2} and the comments below it.

Before going further, let us give some examples for which the above applies. Below we denote by $[\cdot,\cdot]_\beta$ the complex interpolation, see e.g., \cite{BeLo} or \cite[Appendix C]{Analysis1}. 
Of course, many more examples of operators satisfying the assumptions of Theorem \ref{thm:main} are possible, see Subsection \ref{ss:variational_comments} below. 

\begin{example}
\label{ex:operators}
Standard examples of operators $A$ to which Theorem \ref{thm:main} applies are:
\begin{enumerate}[{\rm(1)}]
\item\label{it:strong1} The (strong) shifted Laplacian $\mathrm{Id}-\Delta $ on $L^2(\R^d)$ with domain $W^{2,2}(\R^d)$.
\item\label{it:strong2} The (strong) Dirichlet Laplacian $-\Delta$ on $L^2(\Dom)$ with domain $W^{2,2}(\Dom)\cap W^{1,2}_0(\Dom)$, where $\Dom$ is a bounded $C^2$-domain in $\R^d$.
\end{enumerate}

In the above, $d\geq 1$ is a natural number. Note that in the above situations, the stochastic evolution equation \eqref{eq:SEE} corresponds to a stochastic heat equation with additive noise. 
Before going further, it is worth recalling that the above operators have the so-called \emph{bounded imaginary power} property (often abbreviated as BIP, see \cite[Definition 15.3.4]{Analysis3}). Due to the self-adjointness of the above operators, the BIP property follows again from \cite[Proposition 10.2.23]{Analysis2}. 
For our purposes, it is enough to recall that if the $A$ operator on $X$ has the BIP property, then the domain of the fractional powers can be identified via complex interpolation (see e.g., \cite[Theorem 15.3.9]{Analysis3}):
\begin{equation}
\label{eq:identification_fractional_BIP}
\Do(A^\theta)=[X,\Do(A)]_{\theta} \quad \text{ for all }\theta\in (0,1).
\end{equation}
The above is convenient, as complex interpolation spaces can often be computed explicitly.

For the above cases, Theorem \ref{thm:main} implies the following (see Subsection \ref{ss:notation} below for the unexplained notation).
\begin{enumerate}[{\rm(1$^\prime$)}]
\item $u\in L^2(\Omega;W^{\alpha,2}(0,T;W^{1,2}(\R^d)))$ if and only if $x\in W^{2\alpha,2}(\R^d)$.
\item For $\alpha\neq \frac{1}{4}$, we have $u\in L^2(\Omega;W^{\alpha,2}(0,T;W^{1,2}_0(\Dom)))$ if and only if 
$$
x\in 
\left\{
\begin{aligned}
&W^{2\alpha,2}(\Dom) & \text{ for }& \alpha<\tfrac{1}{4},\\
&W_0^{2\alpha,2}(\Dom) & \text{ for }& \alpha\in (\tfrac{1}{4},\tfrac{1}{2}).
\end{aligned}
\right.
$$ 
In the above, we used that \eqref{eq:identification_fractional_BIP} and \cite{Se} yield
$$
[L^2(\Dom), W^{2,2}(\Dom)\cap W^{1,2}_0(\Dom)]_{\alpha}
=
\left\{
\begin{aligned}
&W^{2\alpha,2}(\Dom) & \text{ for }& \alpha<\tfrac{1}{4},\\
&W_0^{2\alpha,2}(\Dom) & \text{ for }& \alpha\in (\tfrac{1}{4},\tfrac{1}{2});
\end{aligned}
\right.
$$
where $W^{\beta,2}(\Dom)$ denotes the Bessel potential spaces on $\Dom$, and $W^{\beta,2}_0(\Dom)$ subspace of $W^{\beta,2}(\Dom)$ such that $u|_{\partial\Dom}=0$ with $\beta\in (\frac{1}{2},\frac{3}{2})$. In particular, the exceptional case $\alpha\neq \frac{1}{4}$ is due to the application of the results in \cite{Se}. However, Theorem \ref{thm:main} still ensures the equivalence between $u\in L^2(\Omega;W^{\alpha,2}(0,T;W^{1,2}_0(\Dom)))$ and $x\in 
[L^2(\Dom), W^{2,2}(\Dom)\cap W^{1,2}_0(\Dom)]_{1/4}$, but we are not able to compute explicitly the latter.
\end{enumerate}
\end{example}

Before going further, let us mention that in many situations the solution to \eqref{eq:SEE} is known to satisfy mixed space-time regularity results. For instance, from \cite{MaximalLpregularity} it follows that for a large class of operators $A$ (including the positive, invertible, and self-adjoint operator on a Hilbert space \cite[Proposition 10.2.23]{Analysis2}), for progressively measurable processes $g$, one has 
\begin{align}\label{eq:SMR}
g\in L^2(\Omega;L^2(0,T;X))\quad \Longrightarrow\quad  
u\in L^2(\Omega;W^{\alpha,2}(0,T;\Do(A^{1/2-\alpha}))) \  \text{ for all } \alpha\in [0,\tfrac12).
\end{align}
Moreover, in the latter paper, appropriate versions of the above are proven in an $L^p$-setting. For further details on the $L^p$-setting, the reader is referred to \cite[Section 3]{AV25_survey} and the references therein.

\subsection{Counterexample of time regularity in the variational setting}
\label{ss:variational_comments}
Before discussing the implications of Theorem \ref{thm:main} for the conjectured result in \cite[Theorem 1.1]{BH16}, we briefly comment on the variational setting for SPDEs. For a more detailed exposition, the reader is referred to, e.g., \cite{AVvar,KR79,LR15}. Let $V$ and $H$ be Hilbert spaces such that $V\embed H$ densely. Then identifying $H$ with its dual $H'$, it follows that the pairing $\langle u,v\rangle_{V,V'}=(u,v)_{H}$ for $v\in V$ and $u\in H$ uniquely induces an embedding $H\embed V'$ (in other words, $(V,H,V')$ forms a Gelfand triple). 
The variational setting for SPDEs, in particular, includes equations of the form
\begin{equation}
\label{eq:SPDE_variational}
\left\{
\begin{aligned}
&\dd u + \A u \, \dd t = \B(u(t))\,\dd \W(t) \quad \text{ on }[0,T], \\ 
&u(0)=u_0,
\end{aligned}
\right.
\end{equation}
where $\A:V\to V'$ is a positive invertible symmetric variational generator (see e.g., \cite[Subsection A.4.2]{DPZ}), i.e., 
there exist a continuous bilinear form $a:V\times V\to \R$ and a constant $\alpha>0$ such that 
\begin{itemize}
\item $\langle \A u,v\rangle =- a(u,v)$ for all $u,v\in V$.
\item $-a(v,v) \geq  \alpha \|v\|_V^2$ for all $v\in V$.
\item $a(u,v)=a(v,u)$ for all $u,v\in V$.
\end{itemize}
Moreover, the variational setting for SPDEs typically allows for nonlinear diffusions $B$ such that, for all $v,w\in H$,
\begin{equation}
\label{eq:mapping_property_B}
\B(v)\in H\quad \text{ and }\quad 
\|\B(v)- \B (w)\|_{H}\leq N (1+\|v-w\|_H).
\end{equation}
Of course, the assumptions above can be further generalized (see, for example, \cite{AVvar, KR79, LR15}); however, the latter set of assumptions falls within the framework of \cite{BH16}, within which we aim to discuss the implications of our results. 
The important point here is that no further mapping properties on $\B$ are required in \cite{BH16}.
In the previous reference, the authors conjectured  
\begin{align}\label{eq:betterreg}
\tag{C}
u\in L^2(\Omega;W^{\alpha,2}(0,T;V)), \ \ \alpha\in [0,\tfrac12).
\end{align}
The following shows that \eqref{eq:betterreg} cannot hold in a special case that satisfies the above conditions. 

\begin{theorem}\label{thm:main2}
Let $V\embed H\embed V'$ as above, and let $\A:V\to V'$ be a positive invertible symmetric variational generator. Fix $\alpha\in [0,\frac{1}{2})$. Assume that $B(v)\equiv h$ for some $h\in H$.  
Then one has 
$$
u\in L^2(\Omega;W^{\alpha,2}(0,T;V))
\ \ \ \text{ if and only if }\ \ \  h\in [H,V]_{2\alpha} \ \text{\normalfont{(complex interpolation)}}.
$$ 
Moreover, the norm equivalence holds $\big(\E\|u\|_{W^{\alpha,2}(0,T;V)}^2\big)^{1/2}\eqsim_{\alpha,\A,T} \|h\|_{[H,V]_{2\alpha}}$.
\end{theorem}

In particular, the previous result indicates that, in general, one cannot expect time smoothness in the variational setting unless the diffusion operator $\B$ is spatially sufficiently smooth. 
As Example \ref{ex:2} below shows, in some situations and for some value of $\alpha$, time smoothness also requires boundary conditions for $\B$ at the boundary.
Consequently, if $
u\in L^2(\Omega;W^{\alpha,2}(0,T;V))$, then the conditions $(i)$--$(ii)$ at the end of \cite[p.\ 34]{BH16} cannot, in general, be met. 
In the works \cite{W23_siam,W24_stokes}, sharp time-regularity results for (variants of) the stochastic $p$-Laplacian were proven by employing additional regularity assumptions and boundary conditions on the diffusion $\B$. Theorem \ref{thm:main2} ensures that such additional conditions are also necessary for such a result to hold.

\smallskip

In the following, we elaborate on the consequences of Theorem \ref{thm:main2} for the so-called weak Laplace operator, which is a standard example in the theory of variational SPDEs.

\begin{example}
\label{ex:2}
Let $V=W^{1,2}_0(\Dom)$ and $H=L^2(\Dom)$, where $\Dom$ is a bounded Lipschitz domain. It is clear that $V\embed L^2$ and the natural pairing $\langle v,w\rangle =(v,w)_{L^2(\Dom)}=\int_{\Dom} f g \, \dd x $ yields the identification $V'=W^{-1,2}(\Dom)$ (see e.g., \cite[Example 2.1]{AVvar}). In particular, the weak Dirichlet Laplacian can be defined as $\A_D:W^{1,2}_0(\Dom):V\to V'$ given as $\langle \A_D v,w\rangle = -\int_{\Dom} \nabla v \cdot \nabla w \,\dd x $ for all $v,w\in W^{1,2}_0(\Dom)$. Due to the Poincar\'e inequality, $\A_D$ is a symmetric variational generator. It follows from Theorem \ref{thm:main2} that for all $\alpha\in (0,\frac{1}{2})\setminus \frac{1}{4}$, the solution $u$ to \eqref{eq:SPDE_variational} with $\B(u)\equiv x$ satisfies 
$$
u\in L^2(\Omega;W^{\alpha,2}(0,T;W^{1,2}_0(\Dom))) \  \ \ \text{ if and only if } \ \ \ x\in
\left\{
\begin{aligned}
&W^{2\alpha,2}(\Dom), & \text{ for }& \alpha<\tfrac{1}{4},\\
&W_0^{2\alpha,2}(\Dom), & \text{ for }& \alpha\in (\tfrac{1}{4},\tfrac{1}{2}).
\end{aligned}
\right.
$$
As in Example \ref{ex:operators}, the exceptional case $\alpha\neq \frac{1}{4}$ is due to the application of the results in \cite{Se} to compute the complex interpolation $[H,V]_{2\alpha}$, and not a limitation of Theorem \ref{thm:main2}.
\end{example}

\begin{remark}[The case of nonlinear diffusion]
Fix $\alpha\in (0,\frac{1}{2})$. Using a standard bootstrap argument, it is possible to show that the unique solution $u\in L^2(\Omega;L^2(0,T;V))$ to \eqref{eq:SPDE_variational} satisfies $u\in L^2(\Omega;W^{\alpha,2}(0,T;V))
$ if in addition to \eqref{eq:mapping_property_B} one assumes
\begin{equation*}
\B(v)\in [H,V]_{2\alpha} \quad \text{ and }\quad 
\|\B (v)\|_{[H,V]_{2\alpha}}\leq C( 1+\|v\|_{[H,V]_{2\alpha}})
\end{equation*}
for all $v\in [H,V]_{2\alpha}$, where $C>0$ is a constant independent of $v$.
\end{remark}

\subsection{Notation}
\label{ss:notation}
In this subsection, we collect some basic notation, with additional conventions introduced
as needed later in the text.
We write $A \lesssim_{p_1,\dots,p_n} B$ or $A \gtrsim_{p_1,\dots,p_n} B$ in case there exists a constant $C$ depending only on the parameters $p_1,\dots,p_n$ such that $A\leq C B$ or $A\geq C B$. 

\smallskip

\noindent\emph{Probabilistic set-up.} $\W=(\W(t))_{t\geq 0}$ denotes a one-dimensional Brownian motion on a filtered probability space $(\Omega,\A,(\F_t),\mathbf{P})$. Moreover, $\E[X]=\int_{\Omega} X \, \dd \mathbf{P}$ denotes the expectation of a random variable $X$.

\smallskip

\noindent\emph{Function spaces.} For a domain $\Dom\subseteq \R^d$ (note that $\Dom=\R^d$ is allowed), $W^{k,q}(\Dom)$ for  $k\in \N$ and $q\in [1,\infty]$ denote the usual Sobolev spaces. For $\theta\in (0,1)$, we let $W^{\theta,2}(\Dom)$ denote the Sobolev-Slobodeckij space, i.e., the set of $f\in L^2(\Dom)$ such that 
\begin{equation}
\label{eq:fractional_Sobolev_seminorms_Rd}
[f]_{W^{\theta,2}(\Dom)}:=\Big(\int_{\Dom}\int_{\Dom}\frac{|f(x)-f(y)|^2}{|x-y|^{d+2\theta}}\,\dd x \, \dd y\Big)^{1/2}<\infty;
\end{equation}
endowed with the norm 
$$
\|f\|_{W^{\theta,2}(\Dom)}:=
\|f\|_{L^{2}(\Dom)}+
[f]_{W^{\theta,2}(\Dom)}.
$$
We emphasize that the formulation via \eqref{eq:fractional_Sobolev_seminorms_Rd} is not proper in the extreme cases $\theta\in\{ 0,1\}$. For instance, if $\theta=1$, then $W^{1,2}(\Dom)$ is the above-mentioned Sobolev spaces of functions in $f\in L^2(\Dom)$ such that $\nabla f\in W^{1,2}(\Dom;\R^d)$, where derivatives are understood in the distributional sense.

For $\theta\in (0,1)$, we denote by $W^{\theta,2}_0(\Dom)$ the closure of $C^{\infty}_0(\Dom)$ in $W^{\theta,2}(\Dom)$. It is well-known that, for $\theta\in (0,\frac{1}{2})$, the spaces are equivalent $W^{\theta,2}(\Dom)=W^{\theta,2}_0(\Dom)$; while for $\theta\in (\frac{1}{2},\frac{3}{2})$, it holds that $W^{\theta,2}_0(\Dom)$ is the subset of $f\in W^{\theta,2}(\Dom)$ having zero trace on $\partial\Dom$ in the Sobolev sense (see \cite{TayPDE1} or \cite{Tri83} and references therein for further details). Taking into account higher-order traces, analogous results are also available in the case $\theta > \frac{3}{2}$; however, these will not be needed here. 

Finally, for details on vector-valued Sobolev spaces such as the one introduced in \eqref{eq:seminorm_Walpha}, the reader is referred to \cite[Chapter 14]{Analysis3}.

\section{Proof of Theorems \ref{thm:main} and \ref{thm:main2}}
In this section, we prove the main results of this manuscript. We start by proving Theorem \ref{thm:main}.

\subsection{Proof of Theorem \ref{thm:main}}
Let us begin by recalling that under the assumptions of Theorem \ref{thm:main}, the solution to \eqref{eq:SEE} is given by the stochastic convolution
\begin{equation}
\label{eq:stochastic_convolution}
u(t)=\int_0^t e^{-(t-r)A} x \,\dd \W(r),
\end{equation}
where $(e^{-tA})_{t\geq 0}$ denotes the semigroup generated by the operator $-A$.

\begin{proof}[Proof of Theorem \ref{thm:main}]
We begin by noticing $x\in \Do(A^\alpha)$ implies $u\in L^2(\Omega;W^{\alpha,2}(0,T;\Do(A^{1/2})))$ follows from  \eqref{eq:SMR} applied to $g_{\alpha} = A^{\alpha} x$. Indeed, from \eqref{eq:SMR} it follows that 
$$
A^{\frac12} u(t) = A^{\frac12 -\alpha} \int_0^t e^{-(t-r)A} g_{\alpha}(r) \,\dd \W(r)\in L^2(\Omega;W^{\alpha,2}(0,T;X)).
$$  
From the invertibility of $A$, we conclude that $u\in L^2(\Omega;W^{\alpha,2}(0,T;\Do(A^{1/2})))$, and the results of \cite{MaximalLpregularity} also yield the estimate 
$$
\E \|u\|_{W^{\alpha,2}(0,T;\Do(A^{1/2}))}^2 \lesssim_{\alpha,T} \|x\|^2_{\Do(A^{\alpha})}
$$
provided $x\in \Do(A^{\alpha})$. It remains to prove the opposite direction, that is, the main contribution of these notes. In the rest of the proof, we focus on the case $\alpha\in (0,\frac{1}{2})$, which will be fixed below. If $\alpha=0$, then one can use the argument in \eqref{eq:estimate_A12_proof_T} below.  
For the missing implication in Theorem \ref{thm:main}, it is enough to show the following claim.

\smallskip

\emph{Claim: There are constants $c_{\alpha},\theta_{\alpha,T}>0$ such that, for all $x\in \Do(A)$,}
\begin{align}\label{eq:toprove}
\E [u]_{W^{\alpha,2}(0,T;\Do(A^{1/2}))}^2 \geq T c_{\alpha} \|A^{\alpha} x\|^2_{X} - T \theta_{\alpha,T}\|x\|_X^2,
\end{align}
\emph{where $u$ is given via the stochastic convolution \eqref{eq:stochastic_convolution}, and $\lim_{T\to \infty} \theta_{\alpha,T}=0$.}

\smallskip

Before proving the claimed estimate \eqref{eq:toprove}, we first prove that it implies the remaining implication in Theorem \ref{thm:main}. To this end, it is enough to show that the inequality \eqref{eq:toprove} extends to all $x\in\Do(A^\alpha)$. This will be achieved via approximation, see \cite[Proposition 10.1.7]{Analysis2}. To this end, suppose that 
$$u\in L^2(\Omega;W^{\alpha,2}(0,T;\Do(A^{1/2}))).
$$ 
Let $u_n$ be the solution with $x$ replaced by $x_n =n(n+A)^{-1}x$. From \eqref{eq:stochastic_convolution}, it follows that $u_n = n(n+A)^{-1} u$. Since $-A$ is a semigroup generator, it follows that 
$$
\sup_{n\geq 1}\|n(n+A)^{-1}\|_{X\to X}<\infty \qquad \text{ and }\qquad n(n+A)^{-1}y\to y \  \text{ for each }y\in X,
$$ 
see e.g., \cite[Example 10.1.2 and Proposition 10.1.7(1)]{Analysis2}.
In particular, the previous implies $u_n\to u$ in $L^2(\Omega;W^{\alpha,2}(0,T;\Do(A^{\frac12-\alpha})))$ via the dominated convergence theorem and the definition of the Sobolev-Slobodeckij seminorm \eqref{eq:seminorm_Walpha}. Applying \eqref{eq:toprove} to the differences $u_n-u_m$ and $x_n-x_m$ and letting $m,n\to \infty$ gives that the sequence $(A^{\alpha} x_n)_{n\geq 1}$ is Cauchy, and hence convergent. Since $x_n\to x$, the closedness of $A^{\alpha}$ implies that $x\in \Do(A^{\alpha})$ and moreover, \eqref{eq:toprove} extends to $x\in \Do(A^{\alpha})$.

\smallskip
Now, we turn to the proof of \eqref{eq:toprove}.  
Note that for $t>s$, one can write
\begin{align*}
A^{1/2}(u(t) - u(s)) 
&= \int_0^s A^{1/2}(e^{-(t-r)A} - e^{-(s-r)A}) x \, \dd \W(r) \\
&+ \int_s^t A^{1/2} e^{-(t-r)A} x \, \dd \W(r). 
\end{align*}
From the orthogonality in $L^2(\Omega;X)$ of the two terms on the right-hand side of the above, we have 
\begin{align*}
\E\|A^{1/2} (u(t) - u(s))\|_X^2 & \geq \E \Big\|\int_s^t A^{1/2} e^{-(t-r)A} x\, \dd \W(r)\Big\|_X^2 
\\ & = \int_s^t \|A^{1/2} e^{-(t-r)A} x\|_X^2 \,\dd r & \text{(It\^o isometry)}
\\ & = \int_s^t (A e^{-2(t-r)A} x,x)_X \,\dd  r & \text{(self-adjointness)}
\\ & = \int_s^t \frac{1}{2} \Big(\frac{\dd}{\dd r} e^{-2(t-r)A} x,x\Big)_X \,\dd  r &
\\ & =  \frac12 (x - e^{-2(t-s)A}x,x)_X & \text{(fundamental theorem of calculus)}.
\end{align*}
where $(\cdot,\cdot)_X$ denotes the scalar product in $X$. In particular, it holds that $(x - e^{-2(t-s)A}x,x)_X\geq 0$. 

From \eqref{eq:seminorm_Walpha} and the symmetry of the integrand, we have
\begin{align}
\nonumber
\E [u]_{W^{\alpha,2}(0,T;\Do(A^{1/2}))}^2& 
\geq \int_0^T \int_0^t \frac{(x - e^{-2(t-s)A}x,x)_X}{(t-s)^{2\alpha+1}}\, \dd s \,\dd t 
\\
\nonumber
 & =   \int_0^T \int_0^t \frac{(x - e^{-2sA}x,x)_X}{s^{2\alpha+1}} \,\dd s \, \dd t  
\\ \nonumber
& \stackrel{(i)}{\geq} \frac{T}{2} \int_0^{T/2}  \frac{(x - e^{-2sA}x,x)_X}{s^{2\alpha+1}} \,\dd s,
\\ 
\label{eq:final_estimate_lower_bound}
& = \frac{T}{2} \int_0^{\infty}  \frac{(x - e^{-2sA}x,x)_X}{s^{2\alpha+1}}\, \dd s  - 
 \frac{T}{2} \int_{T/2}^\infty  \frac{(x - e^{-2sA}x,x)_X}{s^{2\alpha+1}} \, \dd s,
\end{align} 
where in $(i)$ we used the positivity of the integrand.  
To proceed, recall that $e^{-tA}$ is contractive on $X$. Indeed, for $x\in \Do(A)$, we have $\frac{\dd }{\dd t }\|e^{-tA} x\|_{X}^2 =-2(Ae^{-tA}x,e^{-tA}x)_X\leq 0$ by the self-adjointness and the positivity of $A$. Integrating the latter identity, one obtains $\|e^{-tA}x\|_{X}\leq \|x\|_X$ for all $x\in\Do(A)\subseteq X$. By density, the previous yields contractivity of the semigroup $e^{-tA}$. 
For the second term on the right-hand side of \eqref{eq:final_estimate_lower_bound},
\[ \int_{T/2}^\infty  \frac{(x - e^{-2sA}x,x)_X}{s^{2\alpha+1}} \,\dd s\leq \theta_{\alpha,T}\|x\|_{X}^2,\]
where $\theta_{\alpha,T}\to 0$ as $T\to \infty$. To estimate the first order term in \eqref{eq:final_estimate_lower_bound}, we use the following consequence of Komatsu's formula \cite[Theorem 6.1.6]{MarSanz} for the fractional powers:
$$
\int_0^{\infty}  \frac{x - e^{-sA}x}{s^{\beta+1}}\, \dd s
= K_\beta\, A^\beta x \quad \text{ for all } \ \beta\in (0,1)\  \text{ and }\  x\in \Do(A),
$$
where $K_\beta := \int_0^{\infty}  \frac{1 - e^{-s}}{s^{\beta+1}} \, \dd s$. Note that the above integral is a well-defined Bochner integral in $X$, as $x\in \Do(A)$ and $x - e^{-sA}x=\int_0^s A e^{-rA}x\,\dd r$. Thus, for $x\in \Do(A)$ and $\alpha<\frac{1}{2}$, we find
\begin{align*}
 \int_0^{\infty}  \frac{(x - e^{-2sA}x,x)_X}{s^{2\alpha+1}}\, \dd s & = 2^{2\alpha} \Big(\int_0^{\infty}  \frac{x - e^{-sA}x}{s^{2\alpha+1}}\, \dd s,x\Big)_X  \\
&= 2^{2\alpha} K_{2\alpha}\, (A^{2\alpha}x, x)_X = c_{\alpha} \|A^{\alpha}x\|_X^2,
\end{align*}
where in the last step we used the self-adjointness of $A$ again.

It remains to derive the norm equivalence stated in the theorem. For this, it remains to prove the lower bound $\gtrsim$. We have already proved \eqref{eq:toprove}. Next we will prove a lower bound for $\E\|u\|_{L^2(0,T;\Do(A^{1/2}))}^2$. 
The It\^o isometry and the self-adjointness yield
\begin{align}
\label{eq:estimate_A12_proof_T}
\E\|A^{1/2}u\|_{L^2(0,T;X)}^2 & = \int_0^T \int_0^t \|A^{1/2}e^{-(t-r)A}x\|^2_X \, \dd r \, \dd t 
\\ 
\nonumber
& = \int_0^T \int_0^t (A e^{-2sA}x,x)_X \, \dd s \, \dd t 
\\ 
\nonumber
& = \frac{T}{2} \|x\|^2_X - \frac12 \int_0^T (e^{-2tA} x,x)_X \, \dd t
\\ 
\nonumber
& \geq  \Big(\frac{T}{2}- \frac{1 - e^{-2T\delta}}{4\delta}\Big) \|x\|^2_X= C_{\delta,T}\|x\|^2_X,
\end{align}
where $\delta>0$ satisfies $\|e^{-tA}\|_{X\to X}\leq e^{-\delta t}$, and  $C_{\delta,T}:=\frac{T}{2}- \frac{1 - e^{-2T\delta}}{4\delta}$ (note $C_{T,\delta}>0$
as $1-x<e^{-x}$ for $x>0$). Under our assumptions, the exponential decay of $\|e^{-tA}\|_{X\to X}$ is known to experts. However, for the reader's convenience, we include some details. Recall that $A$ is invertible by assumption, and thus $A^{-1/2}=(A^{-1})^{1/2}$ is bounded (for the fractional powers of bounded operators, see e.g.\ \cite[Chapter 4]{MarSanz}). Let $\delta=\frac{1}{2}\|A^{-1/2}\|_{X\to X}^2$. Note that $((A-\delta)x,x)_X= \|A^{1/2}x\|_{X}^2-\delta\|x\|_X^2\geq \frac{1}{2}\|A^{1/2}x\|_X^2 \geq 0$ for all $x\in \Do(A)$. Since the self-adjointness of $A-\delta$ follows from the one of $A$, the argument below \eqref{eq:final_estimate_lower_bound} applies to $A-\delta$ and yields $\|e^{-t(A-\delta)}\|_{X\to X}\leq 1$.

Therefore, for every $\varepsilon\in (0,1]$,
\begin{align*}
\E\|u\|_{W^{\alpha,2}(0,T;\Do(A^{1/2}))}^2 &= \E[u]_{W^{\alpha,2}(0,T;\Do(A^{1/2}))}^2 + \E\|u\|_{L^2(0,T;\Do(A^{1/2}))}^2 
\\ & \geq \varepsilon \E[u]_{W^{\alpha,2}(0,T;\Do(A^{1/2}))}^2 + \E\|u\|_{L^2(0,T;\Do(A^{1/2}))}^2 
\\ & \geq \varepsilon \big[T c_{\alpha} \|A^{\alpha} x\|^2_{X} - T \theta_{\alpha,T}\|x\|_X^2\big] + C_{\delta,T}\|x\|^2_X
\\ & = \varepsilon T c_{\alpha} \|A^{\alpha} x\|^2_{X}  + (C_{\delta,T} - \varepsilon T \theta_{\alpha,T}) \|x\|^2_X.
\end{align*}
Thus the desired result follows by choosing $\varepsilon = \min\{ \frac{C_{\delta,T}}{2T \theta_{\alpha,T}},1\}$. 
\end{proof}

\subsection{Proof of Theorem \ref{thm:main2}}
In this final subsection, we derive Theorem \ref{thm:main2} from Theorem \ref{thm:main}.

\begin{proof}[Proof of Theorem \ref{thm:main2}]
We begin by recalling that solutions to \eqref{eq:SPDE_variational} exist, and are unique in the class of progressively measurable processes with a.s.\ paths in $L^2(0,T;V)$, see e.g., \cite[Chapter 4]{LR15}. 

Let $\A$ be a positive invertible symmetric variational generator. The operator $\A$ naturally restricts on $H$, leading to an operator $A$, which is defined as 
\begin{itemize}
\item
$\Do(A) = \{v \in V \,:\, a(v,\cdot ) \text{ is continuous in the topology of }H\}.$
\item $ a(v,w)=( A v,w)_H$ for all $v\in \Do(A)$ and $w\in V$.
\end{itemize}
One can check that $A$ is a positive, invertible, and self-adjoint operator on $H$, see e.g., \cite[Proposition A.12]{DPZ}. 
Moreover, it follows from \cite[Proposition 1.10, Chapter 8]{TayPDE2} that 
\begin{equation}
\label{eq:square_root_property_A_variational}
\Do(A^{1/2})=V.
\end{equation}
Recall that $\B(u)\equiv h$ for some $h\in H$. 
Thus, from \eqref{eq:zero_regularity_statement} and the uniqueness of \eqref{eq:SPDE_variational} mentioned at the beginning of this proof, it follows that the unique  \eqref{eq:SPDE_variational} with a.s.\ paths in $L^2(0,T;V)$ is given by the stochastic convolution
$
u(t)=\int_0^t e^{-(t-s)A} h \,\dd \W(s)
$, and therefore is a solution to \eqref{eq:SEE} where $A$ is the induced operator induced by $\A$ on $H$. In particular, Theorem \ref{thm:main} and \eqref{eq:square_root_property_A_variational} ensures that, for all $\alpha\in[0,\frac{1}{2})$,
\begin{equation}
\label{eq:cor_main_proof_2}
u\in L^2(\Omega;W^{\alpha,2}(0,T;V))\ \ \ \text{ if and only if }\ \ \  h\in \Do(A^{\alpha}).
\end{equation}
To conclude, it remains to show that, for all $\alpha\in (0,\frac{1}{2})$,
\begin{equation}
\label{eq:identification_variational_A}
\Do(A^{\alpha})=[H,V]_{2\alpha}.
\end{equation}
As commented above \eqref{eq:identification_fractional_BIP}, as $A$ is self-adjoint, the identification \eqref{eq:identification_fractional_BIP} holds. Thus, from the reiteration of complex interpolation (see \cite[Theorem 4.6.1]{BeLo}), it follows that, for all $\alpha\in (0,\frac{1}{2})$, 
$$
\Do(A^\alpha)\stackrel{\eqref{eq:identification_fractional_BIP}}{=}[H, \Do(A)]_{\alpha}=[H, \Do(A^{1/2})]_{2\alpha}
\stackrel{\eqref{eq:square_root_property_A_variational}}{=}[H, V]_{2\alpha}.
$$
The claim of Theorem \ref{thm:main2} follows by combining the above with \eqref{eq:cor_main_proof_2}.
\end{proof}

\smallskip

\noindent \emph{Acknowledgements}. The authors thank the referee for useful comments and suggestions.

\medskip

\noindent {\bf Declaration of interests.} The authors do not work for, advise, own shares in, or receive funds from any organisation that could benefit from this article, and have declared no affiliation other than their research organisations.

\def\polhk#1{\setbox0=\hbox{#1}{\ooalign{\hidewidth
  \lower1.5ex\hbox{`}\hidewidth\crcr\unhbox0}}} \def\cprime{$'$}

\end{document}